\documentclass{amsart}
\usepackage{amsthm,xcolor}
\usepackage{latexsym,amssymb,amsfonts,amsmath ,graphicx, url, paralist}
\usepackage[vcentermath,enableskew]{youngtab}
\usepackage{color}
\usepackage[noadjust]{cite} 
\usepackage{hyperref, mathrsfs, tikz}
\usetikzlibrary{arrows,decorations.pathreplacing}
\usepackage[margin=1in]{geometry}
\usepackage{bm}
\usepackage{subfig}
\usepackage{mathtools}

\def\aff{{\rm aff}}

\newtheorem{theorem}{Theorem}[section]
\newtheorem{proposition}[theorem]{Proposition}
\newtheorem{lemma}[theorem]{Lemma}

\newtheorem{corollary}[theorem]{Corollary}
\theoremstyle{definition}
\newtheorem{definition}[theorem]{Definition}
\newtheorem{example}[theorem]{Example}
\theoremstyle{remark}
\newtheorem{remark}[theorem]{Remark}

\newcommand{\F}{\mathcal{F}}
\newcommand{\p}{\mathcal{P}}
\newcommand{\s}{\mathfrak{S}}

\DeclareMathOperator{\GT}{GT}

\def\S{\mathfrak{S}}

\newtheoremstyle{named}{}{}{\itshape}{}{\bfseries}{.}{.5em}{#1 \thmnote{#3}}
\theoremstyle{named}
\newtheorem*{namedtheorem}{Theorem}
\newtheorem*{namedlemma}{Lemma}

\newcommand\multiset[2]%
{\mathchoice{\left(\kern-0.5em{\binom{#1}{#2}}\kern-0.5em\right)}
            {\bigl(\kern-0.3em{\binom{#1}{#2}}\kern-0.3em\bigr)}
            {\bigl(\kern-0.3em{\binom{#1}{#2}}\kern-0.3em\bigr)}
            {\bigl(\kern-0.3em{\binom{#1}{#2}}\kern-0.3em\bigr)}}

\title[]{Schubert polynomials as projections of  Minkowski sums of Gelfand-Tsetlin polytopes}
\author{Ricky Ini Liu}
\address{Ricky Ini Liu, Department of Mathematics, North Carolina State University, Raleigh, NC 27695. \newline{riliu@ncsu.edu}}

\author{Karola M\'esz\'aros}
\address{Karola M\'esz\'aros, Department of Mathematics, Cornell University, Ithaca, NY 14853 and School of Mathematics, Institute for Advanced Study, Princeton, NJ 08540.  \newline{karola@math.cornell.edu}
}

\author{Avery St. Dizier}
\address{Avery St. Dizier, Department of Mathematics, Cornell University, Ithaca NY 14853.  \newline{ajs624@cornell.edu}
}

\thanks{Liu  is partially supported by a National Science Foundation Grant (DMS 1758187). 
 M\'esz\'aros is partially supported by a National Science Foundation Grant (DMS 1501059)   as well as by a von Neumann Fellowship at the IAS   funded by the Fund for Mathematics and the Friends of the Institute for Advanced Study.}

\begin{document}

\begin{abstract}
	Gelfand-Tsetlin polytopes are classical objects in algebraic combinatorics arising in the representation theory of $\mathfrak{gl}_n(\mathbb{C})$.  The integer point transform of the Gelfand-Tsetlin polytope $\mathrm{GT}(\lambda)$ projects to the Schur function $s_{\lambda}$. Schur functions form a distinguished basis of the ring of symmetric functions; they are also special cases of Schubert polynomials $\s_{w}$ corresponding to Grassmannian permutations. 
	
	For any permutation $w \in S_n$ with column-convex Rothe diagram, we construct a polytope  $\p_{w}$ whose integer point transform projects to the Schubert polynomial $\s_{w}$. Such a construction has been sought after at least since the construction of twisted cubes by Grossberg and Karshon in 1994, whose integer point transforms project to Schubert polynomials $\s_{w}$ for all $w \in S_n$. However, twisted cubes are not honest polytopes; rather one can think of them as signed polytopal complexes. Our polytope $\p_{w}$ is a convex polytope. We also show that $\p_{w}$ is a Minkowski sum of Gelfand-Tsetlin polytopes of varying sizes. When the permutation $w$ is Grassmannian, the Gelfand-Tsetlin polytope is recovered. We conclude by showing that the Gelfand-Tsetlin polytope is a flow polytope. 
\end{abstract}
\date{\today}
\maketitle

\section{Introduction}
\label{sec:intro}
 
Schubert polynomials, introduced by Lascoux and Sch\"utzenberger in 1982 \cite{LS}, are extensively studied in  algebraic combinatorics \cite{BJS, FK1993, laddermoves, nilcoxeter, thomas, prismtableaux, lenart, manivel, multidegree, KM, sottile}.   They represent cohomology classes of Schubert cycles in flag varieties, and they generalize Schur functions, a distinguished basis of the ring of symmetric functions. 

A well-known property of the Schur function $s_{\lambda}$ is that it is a projection of the integer point transform of the Gelfand-Tsetlin polytope $\GT(\lambda)$.  This has inspired the following natural question for Schubert polynomials:

\medskip

\noindent{\bf Question 1.} {\it For $w\in S_n$, is there a natural polytope $\p_{w}$ and a projection map $\pi_{w}$ such that the projection of the integer point transform of $\p_{w}$ under the map $\pi_{w}$ equals the Schubert polynomial $\S_{w}$?}
 
\medskip

The construction of twisted cubes by Grossberg and Karshon in 1994 \cite{botttowers} is the first attempt at an answer to the above question. The integer point transforms of twisted cubes project to any Schubert polynomial.   Indeed, Grossberg and Karshon show that for both flag and Schubert varieties, their (virtual) characters are projections of integer point transforms of twisted cubes.
The one catch with twisted cubes is that they are not always honest polytopes; intuitively one can think of them as signed polytopal complexes. For the Grassmannian case they do not yield the Gelfand-Tsetlin polytope. Kiritchenko's beautiful work \cite{divdiff} explains how to make certain corrections to the Grossberg-Karshon twisted cubes in order to obtain the  Gelfand-Tsetlin polytope for Grassmannian permutations.

\medskip

Recall that given a partition $\lambda = (\lambda_1,\dots,\lambda_n)\in \mathbb{Z}^n_{\geq 0}$, the \textbf{Gelfand-Tsetlin polytope} $\GT(\lambda)$ is the set of all nonnegative triangular arrays
	\begin{center}
		\begin{tabular}{ccccccc}
			$x_{11}$&&$x_{12}$&&$\cdots$&&$x_{1n}$\\
			&$x_{22}$&&$x_{23}$&$\cdots$&$x_{2n}$&\\
			&&$\cdots$&&$\cdots$&&\\
			&&$x_{n-1,n-1}$&&$x_{n-1,n}$&&\\
			&&&$x_{nn}$&&&
		\end{tabular}
	\end{center}
	such that 
	\begin{align*}
		x_{in}=\lambda_i &\mbox{ for all } 1\leq i\leq n,\\
		x_{i-1,j-1}\geq x_{ij}\geq x_{i-1,j} &\mbox{ for all } 1\leq i \leq j\leq n.
	\end{align*} 

To state our main result, which is a partial answer to  Question 1, we need to consider the Minkowski sums of Gelfand-Tsetlin polytopes of partitions with different lengths.

Fix $n$, and for each $k\in[n]$, let $\lambda^{(k)}$ be a partition with $k$ parts (with empty parts allowed). We wish to study the Minkowski sum 
\[\GT(\lambda^{(1)})+\GT(\lambda^{(2)})+\cdots+\GT(\lambda^{(n)}).\]
To make this Minkowski sum well-defined, we embed $\mathbb{R}^{\binom{k+1}{2}}$ into $\mathbb{R}^{\binom{n+1}{2}}$ for each $k$. To do this, let $y_{ij}$ be coordinates of $\mathbb{R}^{\binom{k+1}{2}}$ and $x_{ij}$ be coordinates of $\mathbb{R}^{\binom{n+1}{2}}$ as in the definition of the Gelfand-Tsetlin polytope. The embedding is given by
\[y_{ij}\mapsto x_{i,j+n-k} \text{ for all }i+j\leq k+1.\]

Given a column-convex diagram $D$ with $n$ rows, we associate to it a family of partitions $\mathrm{Par}_D=\{\lambda^{(1)}, \ldots, \lambda^{(n)}\}$ in the following way. The shape $\lambda^{(i)}$, $i \in [n]$, has $i$ parts and  is obtained from $D$ by ordering the columns of $D$ whose lowest box is in the $i$th row in decreasing fashion and reading off $\lambda^{(i)}$ according to the French notation. Note that  $\lambda^{(i)}$ is empty if there is no column of $D$ whose lowest box is in the $i$th row.

\begin{theorem} 
	\label{thm:gtsum} 
	The character $\mathfrak{s}_D$ of the flagged Schur module associated to a column-convex diagram $D$ with $n$ rows and $\mathrm{Par}_D=\{\lambda^{(1)}, \ldots, \lambda^{(n)}\}$ is a projection of the integer point transform of
\begin{equation} \label{eq:ms} 
	\p_D\coloneqq\mathrm{GT}(\lambda^{(1)})+\mathrm{GT}(\lambda^{(2)})+\cdots+\mathrm{GT}(\lambda^{(n)})
\end{equation}
with the embedding specified above. We obtain $\mathfrak{s}_D(x_i)$ from the integer point transform $\sigma_{\p_D}({x}_{ij})$ via the specialization
\[
x_{ij}\mapsto
\begin{cases}
x_1&\text{when }\,i=1,\\
x_{i-1}^{-1} x_{i}&\text{when }\,i>1.
\end{cases}
\] 
\end{theorem}
 
In the case that $D$ is the Rothe diagram of a permutation $w \in S_n$, the character $\mathfrak{s}_D$ of the flagged Schur module associated to $D$ is the Schubert polynomial $\mathfrak{S}_{w}$. Thus, Theorem \ref{thm:gtsum} answers Question 1 for permutations whose Rothe diagram is column-convex. The necessary background for and the proof of Theorem \ref{thm:gtsum} is in Section \ref{sec:proj}. It is interesting to note that the Newton polytope of a Schubert polynomial is a generalized permutahedron \cite{FMS, MTY}; thus, the affine projection specified in Theorem \ref{thm:gtsum} maps $\p_{D(w)}$ to a generalized permutahedron for column-convex $D(w)$.

Theorem \ref{thm:gtsum} recovers Gelfand-Tsetlin polytopes for Grassmannian permutations.  We conclude our paper by showing  in Theorem \ref{thm:gtflowpolytope} that Gelfand-Tsetlin polytopes are flow polytopes and by showing how to view $\p_D$ in Theorem \ref{thm:gtsum} in the context of flow polytopes. 

\begin{theorem} \label{thm:gtflowpolytope}
	$\GT(\lambda)$ is integrally equivalent to the flow polytope $\mathcal{F}_{G_\lambda}$.
\end{theorem}

Section \ref{sec:gtflowpolytope}  contains the background for and the proof of Theorem \ref{thm:gtflowpolytope}, as well as some of its corollaries.

\section{Polytopes projecting to Schubert polynomials}
\label{sec:proj}

This section is devoted to proving Theorem~\ref{thm:gtsum} and explaining the relevant terminology. We start by defining diagrams, flagged Schur modules, and their characters.

\subsection{Background}

A \textbf{diagram} is a finite subset of $\mathbb{N} \times \mathbb{N}$. Its elements $(i,j)\in D$ are called \textbf{boxes}. We will think of $\mathbb{N}\times \mathbb{N}$ as a grid of boxes in matrix notation, so $(1,1)$ is the topmost and leftmost box. Canonically associated to each permutation is its Rothe diagram.
\begin{definition}
	The \textbf{Rothe diagram} of a permutation $w\in S_n$ is the collection of boxes 
	\[D(w)=\{(i,j) \mid  1\leq i,j\leq n,\, w(i)>j,\, w^{-1}(j)>i \}.\] 
	We can visualize $D(w)$ as the set of boxes remaining in the $n\times n$ grid after crossing out all boxes below or to the right of $(i,w(i))$ for each $i\in [n]$.
\end{definition}

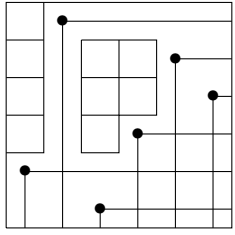
\begin{figure}[ht] 
	\begin{tikzpicture}
	\draw (0,0)--(3,0)--(3,3)--(0,3)--(0,0);
	
	\draw[] 
	(3,2.75)
	-- (.75,2.75) node {$\bullet$}
	-- (0.75,0);
	
	\draw[] 
	(3,2.25)
	-- (2.25,2.25) node {$\bullet$}
	-- (2.25,0);
	
	\draw[] 
	(3,1.75)
	-- (2.75,1.75) node {$\bullet$}
	-- (2.75,0);
	
	\draw[] 
	(3,1.25)
	-- (1.75,1.25) node {$\bullet$}
	-- (1.75,0);
	
	\draw[] 
	(3,.75)
	-- (.25,.75) node {$\bullet$}
	-- (0.25,0);
	
	\draw[] 
	(3,.25)
	-- (1.25,.25) node {$\bullet$}
	-- (1.25,0);
	
	\draw (.5,3)--(.5,1)--(0,1);
	\draw (0,2.5)--(.5,2.5);
	\draw (0,2)--(.5,2);
	\draw (0,1.5)--(.5,1.5);
	
	\draw (1,2.5)--(1,1)--(1.5,1)--(1.5,1.5)--(2,1.5)--(2,2.5)--(1,2.5);
	\draw (1,2)--(2,2);
	\draw (1,1.5)--(1.5,1.5);
	\draw (1.5,2.5)--(1.5,1.5);
	\end{tikzpicture}
	\caption{The permutation $w=256413$ is column-convex and has Rothe diagram \\ $D(w)=\{(1,1),(2,1),(3,1),(4,1),(2,3),(3,3),(4,3),(2,4),(3,4)\}$.}
	\label{fig:rothe}
\end{figure}

\begin{definition}
	A diagram $D$ is \textbf{column-convex} if for each $j$, the set $\{i \mid (i,j)\in D \}$ is an interval in $\mathbb{N}$.
\end{definition}
Note that a Rothe diagram $D(w)$ is column-convex if and only if $w$ avoids the patterns $3142$ and $4132$.

Let $D$ be a diagram with $n$ rows. Denote by $\Sigma_D$ the symmetric group on the boxes in $D$. Let $\mathrm{Col}(D)$ be the subgroup of $\Sigma_D$ permuting the boxes of $D$ within each column, and define $\mathrm{Row}(D)$ similarly for rows. Let $\mathcal{T}_D$ denote the $\mathbb{C}$-vector space with basis indexed by fillings $T\colon D\to [n]$ of $D$. Observe that $\Sigma_D$, $\mathrm{Col}(D)$, and $\mathrm{Row}(D)$ act on $\mathcal{T}_D$ on the right by permuting the filled boxes.

Define idempotents $\alpha_D$, $\beta_D$ in the group algebra $\mathbb{C}[\Sigma_D]$ by
\[
\alpha_D = {1 \over |\mathrm{Row}(D)|} \sum_{w \in \mathrm{Row}(D)} w,\mbox{\hspace{3ex}}
\beta_D = {1 \over |\mathrm{Col}(D)|} \sum_{w \in \mathrm{Col}(D)} \mathrm{sgn}(w) w ,
\]
where $\mathrm{sgn}(w)$ is the sign of the permutation $w$. Given a filling $T\in\mathcal{T}_D$, define $e_T\in\mathcal{T}_D$ to the be the linear combination
\[e_T=T\cdot\alpha_D\beta_D. \]

Identify $\mathcal{T}_D$ with the tensor product $V^{\otimes N}$, where $V=\mathbb{C}^n$ and $N$ is the number of boxes of $D$, in the following manner. First, fix an order on the boxes of $D$. Then read each filling $T$ in this order to obtain a word $i_1,\ldots,i_N$ on $[n]$, and identify this word with the tensor $e_{i_1}\otimes e_{i_2}\otimes\cdots\otimes e_{i_N} \in V^{\otimes N}$, where $e_1,\ldots,e_n$ is the standard basis of $\mathbb{C}^n$. As $GL_n(\mathbb{C})$ acts on $V$, it acts diagonally on $V^{\otimes N}$ by acting on each component. This left action of $GL_n(\mathbb{C})$ on $\mathcal{T}_D$ commutes with the right action of $S_D$. Thus, the subspace of $\mathcal{T}_D$ spanned by all elements $e_T$ is a submodule, called the \textbf{Schur module} of $D$.

Call a filling $T$ of $D$ \textbf{row-flagged} if $T(i,j)\leq i$ for all $i,j$. Let $B_n$ be the subgroup of $GL_n(\mathbb{C})$ consisting of upper triangular matrices. The subspace of $\mathcal{T}_D$ spanned by the elements $e_T$ for $T$ row-flagged forms a $B_n$-submodule of $\mathcal{T}_D$, called the flagged Schur module of $D$.

\begin{definition}
	The \textbf{flagged Schur module} $\mathcal{S}_D$ of a diagram $D$ is the $B_n$-submodule of $\mathcal{T}_D$ spanned by \[\{e_T \mid  T\mbox{ is a row-flagged filling of }D\}.\] 
	The \textbf{formal character} $\mathrm{char}(\mathcal{S}_D)$, denoted by $\mathfrak{s}_D$, is the polynomial 
	\[\mathfrak{s}_D=\mathrm{char}(\mathcal{S}_D)(x_1,\ldots,x_n) = \mathrm{Trace}(X\colon \mathcal{S}_D\to\mathcal{S}_D), \]
	where $X$ is the diagonal matrix in $B_n$ with diagonal entries $x_1,\ldots,x_n$.
\end{definition}

A particularly important subclass of characters of flagged Schur modules is that of Schubert polynomials as explained in Theorem \ref{thm:kp} below. Schubert polynomials are associated to permutations, and they admit various combinatorial and algebraic definitions. For a permutation $w\in S_n$, we will define the Schubert polynomial $\mathfrak{S}_w$ via divided difference operators $\partial_i$ on polynomials. 
\begin{definition}
	The Schubert polynomial of the long word $w_0 \in S_n$ $(w_0(i)=n-i+1$ for $1\leq i\leq n)$ is defined as  \[\mathfrak{S}_{w_0}\coloneqq x_1^{n-1}x_2^{n-2}\cdots x_{n-1}.\]  
	For $w\neq w_0$, there exists $i\in [n-1]$ such that $w(i)<w(i+1)$. For any such~$i$, the \textbf{Schubert polynomial} $\mathfrak{S}_{w}$ is defined by
	\[\mathfrak{S}_{w}\coloneqq\partial_i \mathfrak{S}_{ws_i},\]
	where
	\[\partial_i (f)= \frac{f-s_if}{x_i-x_{i+1}}=
	\frac{f(x_1,\ldots,x_n)-f(x_1,\ldots,x_{i-1},x_{i+1},x_{i},\ldots,x_n)}{x_i-x_{i+1}},\] 
	and $s_i$ is the transposition swapping $i$ and $i+1$. The operators $\partial_i$ can be shown to satisfy the braid relations, so the Schubert polynomials $\mathfrak{S}_{w}$ are well-defined.
\end{definition}
Schubert polynomials appear as the characters of flagged Schur modules of Rothe diagrams.
\begin{theorem}[\cite{KP}] \label{thm:kp}
	Let $w\in S_n$ be a permutation, $D(w)$ be the Rothe diagram of $w$, and $\mathfrak{s}_{D(w)}$ be the character of the associated flagged Schur module $\mathcal{S}_{D(w)}$. Then, 
	\[\mathfrak{S}_w(x_1,\ldots,x_n) = \mathfrak{s}_{D(w)}(x_1,\ldots,x_n). \]
\end{theorem} 

\subsection{Minkowski sums of Gelfand-Tsetlin polytopes}
We now move towards proving Theorem~\ref{thm:gtsum}, which for any column-convex diagram $D$, relates the character $\mathfrak s_D$ with the Minkowski sum \[\p_D=\GT(\lambda^{(1)}) + \cdots + \GT(\lambda^{(n)})\] defined in equation \eqref{eq:ms}. To begin, we describe this Minkowski sum in terms of inequalities. We will need the following Lemma~\ref{lem:gtsum}, which is proved in Section \ref{sec:gtflowpolytope}.

\medskip

\begin{lemma} \label{lem:gtsum}
	If $\lambda$ has $n$ parts, then the Gelfand-Tsetlin polytope $\GT(\lambda)$ decomposes as a Minkowski sum:
	\[\GT(\lambda) = \sum_{k=1}^{n}(\lambda_k-\lambda_{k+1})\mathrm{GT}(1^k0^{n-k}).\]
\end{lemma}

\begin{proposition} \label{prop:gtsum}
Let $\lambda^{(1)}, \dots, \lambda^{(n)}$ be partitions such that $\lambda^{(i)}$ has $i$ (possibly empty) parts. The Minkowski sum $\GT(\lambda^{(1)}) + \cdots + \GT(\lambda^{(n)})$ is defined by the following inequalities:
\begin{itemize}
\item for all $1 \leq i \leq j \leq n$, $x_{i-1,j-1} \geq x_{ij}$; and
\item for any positive integer $k$ and nonempty sequence $I$ of even length $0 \leq i_k < i_{k-1} < \cdots < i_1 < j_1 < j_2 < \cdots < j_k \leq n$, 
\[\sum_{s=1}^k x_{j_s - i_s, j_s} - \sum_{s=1}^{k-1} x_{j_{s+1}-i_s, j_{s+1}} \geq \sum_{s=0}^{i_k} \lambda_{j_1-s}^{(n-s)}, \tag{$*$}\]
with equality when $k=1$ and $j_1=i_1+1$.
\end{itemize}

\end{proposition}

\begin{remark}
A simple calculation shows that if, for instance, $i_{s+1} = i_s$ for some $s$, then neither side of $(*)$ would change if we simply remove $i_{s+1}$ and $j_{s+1}$ from the sequence. Likewise, if $j_s = j_{s+1}$ for some $s$, then neither side would change if we remove $i_s$ and $j_s$ from the sequence. Therefore we may equivalently take the inequalities $(*)$ for sequences $0 \leq i_k \leq \cdots \leq i_1 < j_1 \leq \cdots j_k \leq n$.
\end{remark}

One should observe that the entries occurring on the left side of $(*)$ lie at the corners of a path that zigzags southeast and southwest inside the triangular array.

\begin{example}\label{ex:3rows}
	Suppose $n=3$. We first have inequalities $x_{11} \geq x_{22} \geq x_{33}$ and $x_{12} \geq x_{23}$ as with ordinary Gelfand-Tsetlin patterns. Then for $k=1$, we get equalities
	\[x_{11} = \lambda_1^{(3)}, \qquad x_{12} = \lambda_1^{(2)} + \lambda_2^{(3)}, \qquad x_{13} = \lambda_1^{(1)} + \lambda_2^{(2)} + \lambda_3^{(3)},\]
	as well as inequalities
	\[x_{22} \geq \lambda_2^{(3)}, \qquad x_{23} \geq \lambda_2^{(2)} + \lambda_3^{(3)}, \qquad \text{and} \qquad x_{33} \geq \lambda_3^{(3)}.\]
	Finally, for $k=2$, there is one more inequality, namely
	\[x_{12} - x_{23} + x_{33} \geq \lambda_2^{(3)}.\]
\end{example}

\begin{proof}[Proof of Proposition~\ref{prop:gtsum}]
Let $P = P(\lambda^{(1)}, \dots, \lambda^{(n)}) = \GT(\lambda^{(1)}) + \cdots + \GT(\lambda^{(n)})$, and let $Q$ be the polytope given by the inequalities above, $Q = Q(\lambda^{(1)}, \dots, \lambda^{(n)})$. We first show that $P \subseteq Q$. For any point $(x_{ij})_{1 \leq i \leq j \leq n} \in P$, choose, for each $0 \leq m < n$, points $(y_{ij}^{(n-m)})_{1 \leq i \leq j \leq m} \in \GT(\lambda^{(n-m)})$ summing to it, so that $x_{ij} = \sum_{s = 0}^{j-i} y_{i,j-s}^{(n-s)}$. In particular, $\GT(\lambda^{(n-m)})$ will contribute to a coordinate of the form $x_{j-i,j}$ if and only if $m \leq i$.

Inequalities of the form $x_{i-1,j-1} \geq x_{ij}$ are derived by summing the respective inequalities $y_{i-1,j-1-s}^{(n-s)} \geq y_{i,j-s}^{(n-s)}$ over all $0 \leq s \leq j-i$. For inequalities of type $(*)$, consider a sequence $I$, and suppose first that $0 \leq m \leq i_k$. Then
\[
\sum_{s=1}^k y^{(n-m)}_{j_s-i_s, j_s-m} - \sum_{s=1}^{k-1} y^{(n-m)}_{j_{s+1}-i_s, j_{s+1}-m} = y_{j_1 - i_1, j_1 - m}^{(n-m)} + \sum_{s=1}^{k-1} (y^{(n-m)}_{j_{s+1}-i_{s+1}, j_{s+1}-m} - y^{(n-m)}_{j_{s+1}-i_s,j_{s+1}-m})\geq \lambda^{(n-m)}_{j_1-m},\]
for each term in the sum is nonnegative by the defining inequalities of $\GT(\lambda^{(n-m)})$.
If instead $m>i_k$, then let $k'<k$ be the minimum value such that $m \leq i_{k'}$. Then
\begin{align*}
\sum_{s=1}^{k'} y^{(n-m)}_{j_s-i_s, j_s-m} - \sum_{s=1}^{k'} y^{(n-m)}_{j_{s+1}-i_s, j_{s+1}-m} &= \sum_{s=1}^{k'} (y^{(n-m)}_{j_{s}-i_{s}, j_{s}-m} - y^{(n-m)}_{j_{s+1}-i_s,j_{s+1}-m}) \geq 0
\end{align*}
since again each term in the sum is nonnegative. Summing these inequalities over all $m$ then gives the desired inequality. In the case that $k=1$ and $j_1=i_1+1$, we get equality since
\[x_{1,j_1} = \sum_{s=0}^{j_1-1} y_{1,j_1-s}^{(n-s)} = \sum_{s=0}^{i_1} \lambda_{j_1-s}^{(n-s)}.\]

To show $Q \subseteq P$, we induct on $n$ and then the size of $\lambda^{(n)}$. First suppose $\lambda^{(n)} = \varnothing$. The inequalities involving $x_{jj}$ are $x_{11} \geq x_{22} \geq \cdots \geq x_{nn}$, and, when $i_k = 0$,
\[\sum_{s=1}^{k-1} (x_{j_s-i_s,j_s}-x_{j_{s+1}-i_s,j_{s+1}}) + x_{j_k,j_k} \geq \lambda_{j_1}^{(n)} = 0\]
with equality if also $k=1$ and $j_1 = 1$. These imply that $x_{jj} = 0$ for all $1 \leq j \leq n$ and impose no additional constraints on the other entries. Removing the diagonal of entries $x_{jj}$ then yields a triangular array that satisfies the inequalities defining $Q(\lambda^{(1)}, \cdots, \lambda^{(n-1)})$. Therefore by induction
\[Q(\lambda^{(1)}, \dots, \lambda^{(n-1)}, \varnothing) = Q(\lambda^{(1)}, \dots, \lambda^{(n-1)}) = P(\lambda^{(1)}, \dots, \lambda^{(n-1)}) = P(\lambda^{(1)}, \dots, \lambda^{(n-1)}, \varnothing).\]

If $\lambda^{(n)} \neq \varnothing$, then let $m = \ell(\lambda^{(n)})$ be the number of nonzero parts. We will prove that $Q \subseteq \GT(1^m0^{n-m}) +Q'$, where we let $Q' = Q(\lambda^{(1)}, \dots, \lambda^{(n-1)}, \mu^{(n)})$ for $\mu^{(n)} = (\lambda^{(n)}_1-1, \dots, \lambda^{(n)}_m-1, 0, \dots, 0)$. This will prove the result by induction using Lemma \ref{lem:gtsum}   since then $\GT(1^m0^{n-m}) + \GT(\mu^{(n)}) = \GT(\lambda^{(n)})$.

Recall that Gelfand-Tsetlin polytopes are integral polytopes. Given any integer point  $(x_{ij}) \in Q$, set $t_j = 1$ for $1 \leq j \leq m$, while for $m < j \leq n$, set $t_j$ to be the minimum value such that $t_j > t_{j-1}$ and $x_{t_j-1,j-1} = x_{t_j,j}$ (if such an index exists, otherwise set $t_j = \infty$). Then define the point $(z_{ij})_{1 \leq i \leq j \leq n} \in \GT(1^m, 0^{n-m})$ by $z_{ij} = 1$ if $i \geq t_j$, otherwise $z_{ij}=0$.

We claim that $(x'_{ij}) = (x_{ij}-z_{ij}) \in Q'$. Our choice of $t_j$ guarantees that $x_{i-1,j-1} - x_{i,j} \geq 1$ whenever $z_{i-1,j-1}-z_{i,j} = 1$, which ensures that $x'_{i-1,j-1} \geq x'_{ij}$ for all $1 \leq i \leq j \leq n$. Therefore it suffices to show inequalities of type $(*)$.

Given any sequence $I$, suppose that for some $s$, $z_{j_s-i_{s-1},j_s}=0$ but $z_{j_s-i_s,j_s}=1$. Consider what happens to the left hand side of $(*)$ if we insert $j'=j_{s}-1$ between $j_{s-1}$ and $j_s$, and we insert $i'=j_s-t_{j_{s}}$ between $i_s$ and $i_{s-1}$ to get a new sequence $I'$. (Note that $j_{s-1} \leq j' < j_s$ and $i_s \leq i' < i_{s-1}$.) This reduces the left hand side of $(*)$ by
 
\begin{align*}
(x'_{j' - i_{s-1},j'}-x'_{j_s-i_{s-1},j_s}) - (x'_{j'-i',j'}-x'_{j_s-i',j_s}) &= 
(x'_{j_s-1-i_{s-1},j_s-1}-x'_{j_s-i_{s-1},j_s})-(x'_{t_{j_s}-1,j_s-1}-x'_{t_{j_s},j_s})\\
&= x'_{j_s-1-i_{s-1},j_s-1}-x'_{j_s-i_{s-1},j_s}\\ &\geq 0,
\end{align*}
while the right hand side of $(*)$ is unchanged. Thus $(*)$ for the sequence $I$ is implied by $(*)$ for the new sequence $I'$. Since $z_{j_s-i',j_s} = z_{t_{j_s},j_s} = 1$, by iteratively applying this procedure to the new sequence, we will eventually arrive at a sequence for which such an $s$ does not exist.

It therefore suffices to prove inequality $(*)$ in the case that there exists some $s'$ such that $z_{j_s-i_{s-1},j_s}=1$ and  $z_{j_s-i_s,j_s}= 1$ exactly when $s \leq s'$. If $j_1 \leq m$, then the left hand side of $(*)$ is
\begin{align*}
\sum_{s=1}^k x'_{j_s - i_s, j_s} - \sum_{s=1}^{k-1} x'_{j_{s+1}-i_s, j_{s+1}} &= \left(\sum_{s=1}^k x_{j_s - i_s, j_s}-s'\right) - \left(\sum_{s=1}^{k-1} x_{j_{s+1}-i_s, j_{s+1}}-s'+1\right)\\
&= \sum_{s=1}^k x_{j_s - i_s, j_s} - \sum_{s=1}^{k-1} x_{j_{s+1}-i_s, j_{s+1}} - 1,
\end{align*}
while the right hand side is
\[\mu_{j_1}^{(n)} + \sum_{s=1}^{i_k} \lambda_{j_1-s}^{(n-s)} = \sum_{s=0}^{i_k} \lambda_{j_1-s}^{(n-s)} -1,\]
so this inequality follows from the corresponding inequality for $(x_{ij}) \in Q$. If $j_1 > m$, then consider the sequence obtained by inserting $m, m+1, \dots, j_1-1$ before $j_1$, and $j_1-t_{j_1}, j_1-1-t_{j_1-1}, \dots, m+1-t_{m+1}$ after $i_1$ in the sequence. For $(x_{ij}) \in Q$, this yields the inequality
\[\left(\sum_{j = m}^{j_1-1} x_{t_{j+1}-1,j}+\sum_{s=1}^k x_{j_s-i_s, j_s}\right) - \left(\sum_{j =m}^{j_1-1} x_{t_{j+1},j+1} + \sum_{s=1}^{k-1} x_{j_{s+1}-i_s,j_{s+1}}\right) \geq \sum_{s=0}^{i_k} \lambda_{m-s}^{(n-s)}.\]
But $x_{t_{j+1}-1,j} = x_{t_{j+1},j+1}$, and the right side is strictly greater than $\sum_{s=0}^{i_k} \lambda_{j_1-s}^{(n-s)}$ (since $\lambda^{(n)}_{m} > 0 = \lambda^{(n)}_{j_1}$). Thus
\[\sum_{s=1}^k x_{j_s-i_s, j_s} - \sum_{s=1}^{k-1} x_{j_{s+1}-i_s,j_{s+1}} \geq  \sum_{s=0}^{i_k} \lambda_{j_1-s}^{(n-s)} + 1,\]
or equivalently,
\[\left(\sum_{s=1}^k x_{j_s-i_s, j_s}-s'\right) - \left(\sum_{s=1}^{k-1} x_{j_{s+1}-i_s,j_{s+1}}-s'+1\right) \geq  \sum_{s=0}^{i_k} \lambda_{j_1-s}^{(n-s)} = \mu_{j_1}^{(n)} + \sum_{s=1}^{i_k} \lambda_{j_1-s}^{(n-s)},\]
which is the inequality $(*)$ for $(x'_{ij}) \in Q'$. This completes the proof.
\end{proof}

\subsection{Demazure operators and parapolytopes}

To prove Theorem~\ref{thm:gtsum}, we will need a formula for the character $\mathfrak s_D$.  The following formula is essentially a particular case of one due to Magyar \cite{magyar}. (See also Reiner-Shimozono \cite{dpeel}.) We first define the \textbf{isobaric divided difference operator} (or \textbf{Demazure operator}) $\pi_i$ acting on polynomials $f(x_1, \dots, x_n)$ by
\[\pi_if(x_1, \dots, x_n) = \partial_i (x_if)=\frac{x_if - x_{i+1}s_if}{x_i-x_{i+1}},\]
where $s_if$ is the polynomial obtained from $f$ by switching $x_i$ and $x_{i+1}$. Note that $\pi_i f = f$ if $f$ is symmetric in $x_i$ and $x_{i+1}$.

\begin{proposition} \label{prop:sd}
	 Let $D$ be a column-convex diagram with $n$ rows with ${\rm{Par}}_D=\{\lambda^{(1)}, \ldots, \lambda^{(n)}\}$. Define $\widetilde D$ to be the diagram with $n-1$ rows such that $\mathrm{Par}_{\widetilde D} = (\widetilde \lambda^{(1)}, \dots, \widetilde \lambda^{(n-1)})$, where $\widetilde \lambda^{(i)}_j = \lambda^{(i+1)}_j - \lambda^{(i+1)}_{i+1}$. (Here, $\widetilde D$ is obtained from $D$ by removing any column with a box in the first row and then shifting all remaining boxes up by one row.) Also let
	 \[\mu = (\lambda^{(1)}_1 + \lambda^{(2)}_2+\cdots + \lambda^{(n)}_n, \lambda^{(2)}_2 + \cdots + \lambda^{(n)}_n, \dots, \lambda^{(n)}_n),\]
	 the partition formed from all columns of $D$ with a box in the first row. Then
	 \[\mathfrak s_{D} = x_1^{\mu_1}\cdots x_n^{\mu_n} \pi_1\pi_2\cdots \pi_{n-1} (\mathfrak s_{\widetilde D}).\]
\end{proposition}
\begin{proof}
	Note that $D$ can be obtained from $\widetilde D$ by switching the $i$th and $(i+1)$st row for $i = n-1, n-2, \dots, 1$, and then adding $\mu_i$ columns with boxes in rows $\{1, \dots, i\}$ for each $i = 1, \dots, n$. The result then follows immediately from \cite{magyar} (see, for instance, Proposition 15).
\end{proof}

We now show that the polytope for $D$ can be constructed iteratively in a way that mimics the application of the operator $\pi_i$. This geometric operation is the same as the operator $D_i$ given by Kiritchenko in \cite{divdiff} specialized for our current situation.

The key lemma is the following calculation.
\begin{lemma} \label{lem:Di}
Choose nonnegative integers $N_1$, $N_2$ and $\mu_1 \leq \nu_1$, \dots, $\mu_k \leq \nu_k$ such that $\sum_{i=1}^k (\mu_i + \nu_i) \leq N_1+N_2$. Define the polynomial
\[f(x_1, x_2) = \sum_{c_1 = \mu_1}^{\nu_1} \cdots \sum_{c_k = \mu_k}^{\nu_k} x_1^{N_1 - c_1 - \cdots - c_k}x_2^{c_1 + \cdots + c_k-N_2} .\]
Then
\[\pi_1 f(x_1, x_2) = \sum_{c_1 = \mu_1}^{\nu_1} \cdots \sum_{c_k = \mu_k}^{\nu_k} \sum_{c_{k+1} = 0}^{\nu_{k+1}}x_1^{N_1 - c_1 - \cdots - c_k - c_{k+1}}x_2^{c_1 + \cdots + c_k + c_{k+1}-N_2} ,\]
where $\nu_{k+1} = N_1+N_2 - \sum_{i=1}^k (\mu_i + \nu_i)$.
\end{lemma}
\begin{proof}
Note that reversing the order of each of the summations in the expression for $f$ gives
\[f = \sum_{c_1 = \mu_1}^{\nu_1} \cdots \sum_{c_k = \mu_k}^{\nu_k} x_1^{\nu_{k+1}-N_2+c_1 + \cdots + c_k}x_2^{N_1-\nu_{k+1}-c_1-\cdots - c_k}  = (\tfrac{x_1}{x_2})^{\nu_{k+1}} \cdot s_1f.\]
Hence
\[\pi_1f = \frac{x_1f - x_2s_1f}{x_1-x_2} = f \cdot \frac{1 - \left(\tfrac{x_2}{x_1}\right)^{\nu_{k+1}+1}}{1-\tfrac{x_2}{x_1}} = f \cdot \sum_{c_{k+1}=0}^{\nu_{k+1}} x_1^{-c_{k+1}}x_2^{c_{k+1}},\]
as desired.
\end{proof}

Consider $\mathbb R^{\binom{n+1}{2}}$ with coordinates $x_{ij}$ for $1 \leq i \leq j \leq n$. Let $\varphi_k \colon \mathbb R^{\binom{n+1}{2}} \to \mathbb R^{\binom{n+1}{2} - n-1+k}$ be the projection onto the coordinates $x_{ij}$ for all $i \neq k$.

\begin{definition}[\cite{divdiff}]\label{def:parapolytope}
A \textbf{parapolytope} $P \subset \mathbb R^{\binom{n+1}{2}}$ is a convex polytope such that, for all $k$, every fiber of the projection $\varphi_k$ on $P$ is a coordinate parallelepiped.

In other words, for every $k$ and every set of constants $c_{ij}$ ($i \neq k$), there exist constants $\mu_j$ and $\nu_j$ (depending on the $c_{ij}$) such that $(x_{ij}) \in P$ with $x_{ij} = c_{ij}$ for $i \neq k$ if and only if $\mu_j \leq x_{kj} \leq \nu_j$.

We denote this parallelepiped (which depends on $k$ and $c_{ij}$ for $i \neq k$) by \[\Pi(\mu_k, \dots, \mu_n; \nu_k, \dots, \nu_n)=\Pi(\mu,\nu) = \{(x_{kj})_{j=k}^n \mid \mu_j \leq x_{kj} \leq \nu_j\} \subset \mathbb R^{n+1-k}.\]
\end{definition}

Given a polytope $P \subset \mathbb R^{\binom{n+1}{2}}$, let $\sigma_P$ be its {\bf integer point transform}
\[\sigma_P(x_{ij}) = \sum_{(c_{ij}) \in P \cap \mathbb Z^{\binom{n+1}{2}}} \prod_{1 \leq i \leq j \leq n} x_{ij}^{c_{ij}},\]
and define $s_P(x_i)$ to be the image of $\sigma_P(x_{ij})$ under the specialization sending
\[
x_{ij}\mapsto
\begin{cases}
x_1&\text{when }\,i=1,\\
x_{i-1}^{-1} x_{i}&\text{when }\,i>1.
\end{cases}
\]
In other words, the point $(c_{ij}) \in P \cap \mathbb Z^{\binom{n+1}{2}}$ corresponds to the monomial in which the exponent of $x_i$ is $C_i - C_{i+1}$, where $C_i = \sum_{j=i}^n c_{ij}$.

\begin{lemma}\label{lem:Di2}
	Fix $2 \leq k \leq n$, and let $P, Q \subset \mathbb R^{\binom{n+1}{2}}$ be parapolytopes. Suppose that for any fixed integer point $c = (c_{ij})_{i \neq k}$, the fiber over $c$ of the projection $\varphi_k$ on $P$ is the (integer) parallelepiped
	\[\Pi_P = \Pi(\mu_k, \dots, \mu_{n-1}, 0; \nu_k, \dots, \nu_{n-1}, 0),\]
	while the fiber over $c$ of $\varphi_k$ on $Q$ is
	\[\Pi_Q =\Pi(\mu_k, \dots, \mu_{n-1}, \mu_n; \nu_k, \dots, \nu_{n-1}, \nu_n),\]
	where $\mu_n = 0$ and  \[\nu_n = \sum_{j=k-1}^n c_{k-1,j} + \sum_{j=k+1}^n c_{k+1,j} - \sum_{j=k}^{n-1} (\mu_j+\nu_j) \geq 0.\]
	Then $s_Q = \pi_{k-1} s_P$.
\end{lemma}

\begin{proof}
	For fixed $c$, the contribution to $s_P$ of the fiber over $c$ has the form \[M \cdot \sum_{(c_{kk}, \dots, c_{k,n-1}) \in \Pi_P} x_{k-1}^{C_{k-1} - \sum_j c_{kj}} x_k^{\sum_j c_{kj} - C_{k+1}},\]
	where $M$ is a monomial that does not contain $x_{k-1}$ nor $x_k$, and $C_i = \sum_{j=i}^n c_{ij}$ only depends on $c$ for $i \neq k$. This summation has the same form as the one in Lemma~\ref{lem:Di}, so applying $\pi_{k-1}$ as per the lemma immediately gives the result.
\end{proof}

\begin{remark} \label{rem:para}
	The operator that produces $Q$ from $P$ is denoted by $D_{k-1}$ in \cite{divdiff}. However, it is important to note that the operator $D_{k-1}$ will not in general yield a parapolytope or even necessarily a polytope from a general parapolytope $P$.
\end{remark}

We are now ready to prove our main theorem.

\begin{proof} [Proof of Theorem~\ref{thm:gtsum}]
	Let $D$ be a column-convex diagram with $n$ rows with $\mathrm{Par}_D = \{\lambda^{(1)}, \dots, \lambda^{(n)}\}$, and let
	\[\p_D = \GT(\lambda^{(1)}) + \cdots + \GT(\lambda^{(n)}).\] We first claim that we can reduce to the case when $D$ does not contain any boxes in the first row. Indeed, adding a column with boxes in rows $1, 2, \dots,k$ to $D$ serves to add $1$ to each part of $\lambda^{(k)}$, which, by Lemma \ref{lem:gtsum}, translates $\GT(\lambda^{(k)})$ by the single point $\GT(1^k)$ and hence does the same to $\p_D$. This translation adds $k+1-i$ to the sum of row $i$ for $i=1, \dots, k$, so it multiplies $s_{\p_D}$ by $x_1x_2 \cdots x_k$. Since Proposition~\ref{prop:sd} shows that adding this column also multiplies $\mathfrak s_D$ by $x_1x_2 \cdots x_k$, the claim follows.
	
	Therefore, we may assume that $D$ has no boxes in the first row, so that $\lambda^{(k)}_k = 0$ for all $k$, which implies that $\p_D$ is contained in the hyperplane $x_{1n} = 0$. Denote by $\p_D^{(m)}$ the intersection of $\p_D$ with the subspace $x_{1n} = x_{2n} = \cdots = x_{mn} = 0$. In fact, $\p_D^{(m)}$ is also the orthogonal projection of $\p_D$ onto this subspace. To see this, note that for any Gelfand-Tsetlin pattern $(y_{ij})_{1 \leq i \leq j \leq k} \in \GT(\lambda^{(k)})$, setting $y_{in} = 0$ for any $i \leq m$ again yields a valid Gelfand-Tsetlin pattern. Thus for any $(x_{ij}) \in \p_D$, setting $x_{in} = 0$ for all $i \leq m$ will again yield an element of $\p_D$.
	
	Note also that $\p_D^{(n)}$ is just a translate of $\F_{\widetilde D}$ where $\widetilde D$ is the diagram obtained from $D$ by shifting each box up by one row. (Any Gelfand-Tsetlin pattern for $\lambda^{(k)}$ in which the last entry in each row is $0$ is just a Gelfand-Testlin pattern for $\lambda^{(k)}$, thought of as a partition of length $k-1$.) It follows that $\mathfrak s_{\widetilde D} = s_{\F_{\widetilde D}} = s_{\p_D^{(n)}}$.
	
	We first show that the inequalities defining $\p_D^{(m)}$ are precisely the inequalities for $\p_D$ described in Proposition~\ref{prop:gtsum} that do not involve any $x_{in}$ for $i \leq m$ (together with $x_{1n} = x_{2n} = \cdots = x_{mn} = 0$). Clearly any inequality of the form $x_{i-1,n-1} \geq x_{in}$ for $i \leq m$ is redundant since it is implied by inequality $(*)$ for $i_1 = n-i$ and $j_1 = n-1$. Then consider any inequality $(*)$ for a sequence $I$ with $j_k = n$ and $i_{k-1} \geq n-m$:
	\[\sum_{s=1}^{k-1} x_{j_s - i_s, j_s} - \sum_{s=1}^{k-2} x_{j_{s+1}-i_s, j_{s+1}}  + x_{n-i_k, n} - x_{n-i_{k-1}, n}\geq \sum_{s=0}^{i_k} \lambda_{j_1-s}^{(n-s)}.\]
	Let $I'$ be the sequence obtained from $I$ by removing $i_k$ and $j_k$. The corresponding inequality is
	\[\sum_{s=1}^{k-1} x_{j_s - i_s, j_s} - \sum_{s=1}^{k-2} x_{j_{s+1}-i_s, j_{s+1}} \geq \sum_{s=0}^{i_{k-1}} \lambda_{j_1-s}^{(n-s)}.\]
	Since $x_{n-i_k,n} \geq 0$, $x_{n-i_{k-1},n} = 0$, and $i_{k-1} > i_k$, we see that the inequality for $I$ follows immediately from that for $I'$.
	
	Since none of the inequalities defining $\p_D^{(m)}$ involve two coordinates in the same row, $\p_D^{(m)}$ is a parapolytope. It therefore suffices to show that $\p_D^{(m)}$ and $\p_D^{(m-1)}$ are related as in Lemma~\ref{lem:Di2}, for it will then follow that $s_{\p_D^{(m-1)}} = \pi_{m-1}s_{\p_D^{(m)}}$, which combined with $\mathfrak s_{\widetilde D} = s_{\p_D^{(n)}}$ and $s_{\p_D} = s_{\p_D^{(1)}} $ will imply that $s_{\p_D} = \pi_1\pi_2 \cdots \pi_{n-1}(\mathfrak s_{\widetilde D}) = \mathfrak s_D$ by Proposition~\ref{prop:sd}, as desired.
	
	Therefore, fix $c_{ij}$ for $i \neq m$, with $c_{in} = 0$ for $i < m$, and define $\mu_m, \dots, \mu_{n}, \nu_m, \dots, \nu_{n}$ as in Definition~\ref{def:parapolytope} for $\p_D^{(m-1)}$. We claim that $\nu_j + \mu_{j-1} = c_{m-1,j-1} + c_{m+1,j}$. It will then follow by summing over all $j$ that \[\nu_n = \sum_{j=m-1}^{n-1} c_{m-1,j} + \sum_{j=m+1}^{n} c_{m+1,j} - \sum_{j = m}^{n-1}(\mu_j + \nu_j).\]
	Together with noting that the only lower bound on $x_{mn}$ is $0$, this will complete the proof by Lemma~\ref{lem:Di2}.
		
	Consider the upper bounds on $x_{mj}$ in $\p_D^{(m-1)}$. We need to show that if $x_{mj} \leq C$ (where $C$ is some function of $c_{ij}$ for $i \neq m$), then $x_{m,j-1} \geq c_{m-1,j-1} + c_{m+1,j} - C$. This is immediate for the inequality $x_{mj} \leq c_{m-1,j-1}$ since $x_{m,j-1} \geq c_{m+1, j}$. Then consider a sequence $I$ such that $j_{s'+1}-i_{s'} = m$ and $j_{s'+1} = j$ for some $s'$, so that $-x_{mj}$ appears on the left side of $(*)$. Thus $C-x_{mj} \geq 0$, where
	\[C = \sum_{s=1}^k c_{j_s-i_s,j_s} - \sum_{\substack{1 \leq s \leq k-1\\s \neq s'}} c_{j_{s+1-i_s,j_{s+1}}} - \sum_{s=0}^{i_k} \lambda_{j_1-s}^{(n-s)}.\]
	 By inserting $j_{s'+1}-1 = j-1$ before $j_{s'+1} = j$ and $i_{s'}-1 = j-m-1$ before $i_{s'} = j-m$ in $I$ to get a new sequence $I'$, the left side of $(*)$ for $I'$ differs from the left side of $(*)$ for $I$ by $x_{m,j-1} + x_{m,j} - c_{m-1,j-1} - c_{m+1,j}$. Therefore the inequality $(*)$ for $I'$ is equivalent to
	 \[C+x_{m,j-1} - c_{m-1,j-1} - c_{m+1,j} \geq 0,\]
	 or $x_{m,j-1} \geq c_{m-1,j-1} + c_{m+1,j} - C$, as desired. A similar argument shows that any lower bound $x_{m,j-1} \geq C'$ yields an upper bound $x_{mj} \leq c_{m-1,j-1} + c_{m+1,j} - C'$, which completes the proof.
\end{proof}

\begin{example} \label{ex:3rows2}
	Let $n=3$, and let $D$ be the column-convex diagram shown below with  $\lambda^{(3)} = (a+b,a,0)$, $\lambda^{(2)} = (c,0)$, and $\lambda^{(1)} = (0)$.
	\[
	\begin{tikzpicture}[scale=.5]
	\node at (0,2) {$1$};
	\node at (0,1) {$2$};
	\node at (0,0) {$3$};
	\draw (1,-.5) rectangle (2,.5) rectangle (1,1.5);
	\node at (3,0){$\cdots$};
	\node at (3,1){$\cdots$};
	\draw (4,-.5) rectangle (5,.5) rectangle (4,1.5);
	\draw (2,-.5) rectangle (4,.5) rectangle (2,1.5);
	\draw[decorate,decoration={brace,amplitude=5pt,raise=1ex}] (4.9,-.5)--(1.1,-.5) node[midway,yshift=-2em]{$a$};
	\draw (5,-.5) rectangle (9,.5) (6,-.5)--(6,.5) (8,-.5)--(8,.5);
	\node at (7,0){$\cdots$};
	\draw[decorate,decoration={brace,amplitude=5pt,raise=1ex}] (8.9,-.5)--(5.1,-.5) node[midway,yshift=-2em]{$b$};
	\draw (9,.5) rectangle (13,1.5) (10,.5)--(10,1.5) (12,.5)--(12,1.5);
	\node at (11,1){$\cdots$};
	\draw[decorate,decoration={brace,amplitude=5pt,raise=1ex}] (12.9,-.5)--(9.1,-.5) node[midway,yshift=-2em]{$c$};
	\end{tikzpicture}
	\]
	Using the notation in the proof of Theorem~\ref{thm:gtsum}, all the polytopes $\p_D^{(m)}$ for $m=1,2,3$ have $x_{11} = a+b$, $x_{12} = a+c$, and $x_{13} = 0$.
	\begin{itemize}
		\item For $m=3$, $\p_D^{(3)}$ is a segment since we have $a \leq x_{22} \leq a+b$.
		\item For $m=2$, the fiber of $\p_D^{(2)}$ above a point of $\p_D^{(3)}$ is defined by $0 \leq x_{33} \leq x_{22}$, making $\p_D^{(2)}$ a trapezoid. Note that for fixed $x_{33}$, the condition on $x_{22}$ is that $\max\{a, x_{33}\} \leq x_{22} \leq a+b$.
		\item For $m=1$, the fiber of $\p_D = \p_D^{(1)}$ above a point of $\p_D^{(2)}$ is defined by
		\begin{align*}
		0 \leq x_{23} &\leq x_{11} + x_{12} + x_{13} + x_{33} - (\mu_2 + \nu_2)\\
		&= (a+b)+(a+c) +0 + x_{33} - (\max\{a, x_{33}\} + a+b)\\
		&= c+ \min\{a, x_{33}\}.
		\end{align*}
		This is equivalent to the inequalities on $x_{23}$ given in Example~\ref{ex:3rows}:
		\begin{align*}
		\lambda^{(2)}_2+\lambda^{(3)}_3 = 0 \leq x_{23} &\leq c+a = x_{12},\\
		x_{23} &\leq c+x_{33} = x_{12}- \lambda_2^{(3)} + x_{33}.
		\end{align*}
	\end{itemize}
	See Figure~\ref{fig:3rows} for a depiction of $\p^{(m)}_D$ for $m=3,2,1$.
	\begin{figure}
		\begin{tikzpicture}[z = {(250:.6)}]
			\filldraw[fill=blue!10] (0,0,2)--(2,0,2)--(4,0,4)--(0,0,4);
			\node at (1.5,0,3){$\p_D^{(2)}$};
			\node (1) at (-1.5,0,2){$\p_D^{(3)}$};
			\draw[->,>=stealth] (1)--(-.1,0,2.7);
			\draw[ultra thick, red] (0,0,2)--(0,0,4);
			\draw (0,2,2)--(2,4,2)--(2,4,4)--(0,2,4)--cycle (2,4,2)--(4,4,4)--(2,4,4);
			\draw[gray!50] (0,0,2)--(0,2,2) (2,0,2)--(2,4,2);
			\draw  (0,0,4)--(0,2,4) (4,0,4)--(4,4,4);
			\node at (1.5,2,3){$\p_D^{(1)}$};
		\end{tikzpicture}
		\caption{$\p_D=\p^{(1)}_D$ with faces $\p^{(2)}_D$ and $\p^{(3)}_D$ as in Example~\ref{ex:3rows2}. (See also Example~\ref{ex:3rows}.)}
		\label{fig:3rows}
	\end{figure}
\end{example}

\begin{remark}
	The results of Magyar \cite{magyar} allow one to compute the character of the flagged Schur module for any diagram whose columns form a so-called \emph{strongly separated family} (or equivalently, for any \emph{percentage-avoiding diagram} \cite{dpeel}), which includes all Rothe diagrams of permutations. The technique above can be used to find suitable polytopes for a somewhat more general class of diagrams and permutations as Minkowski sums of faces of Gelfand-Tsetlin polytopes (such as the intermediate steps $\p_D^{(m)}$ in the proof of Theorem~\ref{thm:gtsum}), but it does not apply in full generality to all Schubert polynomials due to the ill behavior of general parapolytopes (see Remark~\ref{rem:para}).
\end{remark}

\section{Gelfand-Tsetlin polytopes as flow polytopes}
\label{sec:gtflowpolytope}

In this section we show that the Gelfand-Tsetlin polytope is integrally equivalent to  a flow polytope and give alternative proofs of several known results using flow polytopes. We start by defining flow polytopes and providing the necessary background on them. 

\subsection{Background on flow polytopes}

Let $G$ be a loopless directed acyclic connected (multi-)graph on the vertex set $[n+1]$ with $m$ edges. An integer vector $a=(a_1,\ldots,a_n,-\sum_{i=1}^na_i)\in \mathbb{Z}^{n+1}$ is called a \textbf{netflow vector}. A pair $(G,a)$ will be referred to as a \textbf{flow network}. To minimize notational complexity, we will typically omit the netflow $a$ when referring to a flow network $G$, describing it only when defining $G$. When not explicitly stated, we will always assume vertices of $G$ are labeled so that $(i,j)\in E(G)$ implies $i<j$.

To each edge $(i,j)$ of $G$, associate the type $A$ positive root $e_i-e_j\in\mathbb{R}^n$. Let $M_G$ be the incidence matrix of $G$, the matrix whose columns are the multiset of vectors $e_i-e_j$ for $(i,j)\in E(G)$. A \textbf{flow} on a flow network $G$ with netflow $a$ is a vector $f=(f(e))_{e\in E(G)}$ in $\mathbb{R}_{\geq 0}^{E(G)}$ such that $M_Gf=a$. Equivalently, for all $1\leq i \leq n$, we have 
\[\sum_{e=(k,i)\in E(G)}f(e)+a_i = \sum_{e=(i,k)\in E(G)} f(e). \]
The fact that the netflow of vertex $n+1$ is $-\sum_{i=1}^n a_i$ is implied by these equations.

Define the \textbf{flow polytope} $\mathcal{F}_G(a)$ of a graph $G$ with netflow $a$ to be the set of all flows on $G$:
\[\mathcal{F}_G=\mathcal{F}_G(a)=\{f\in\mathbb{R}^{E(G)}_{\geq 0} \mid M_Gf=a \}. \]
 
\begin{remark}
	\label{rem:flownetwork}
	When $G$ is a flow network $(G,a)$, we will write $\mathcal{F}_G$ for $\mathcal{F}_G(a)$. 
\end{remark}
 
\subsection{The Gelfand-Tsetlin polytope as a flow polytope} 

\begin{namedtheorem}[\ref{thm:gtflowpolytope}]
	$\mathrm{GT}(\lambda)$ is integrally equivalent to $\mathcal{F}_{G_\lambda}$.
\end{namedtheorem}

\medskip

Recall that given a partition $\lambda = (\lambda_1,\dots,\lambda_n)\in \mathbb{Z}^n_{\geq 0}$, the \textbf{Gelfand-Tsetlin polytope} $\mathrm{GT}(\lambda)$ is the set of all nonnegative triangular arrays
\begin{center}
	\begin{tabular}{ccccccc}
		$x_{11}$&&$x_{12}$&&$\cdots$&&$x_{1n}$\\
		&$x_{22}$&&$x_{23}$&$\cdots$&$x_{2n}$&\\
		&&$\cdots$&&$\cdots$&&\\
		&&$x_{n-1,n-1}$&&$x_{n-1,n}$&&\\
		&&&$x_{nn}$&&&
	\end{tabular}
\end{center}
such that 
\begin{align*}
x_{in}=\lambda_i &\mbox{ for all } 1\leq i\leq n\\
x_{i-1,j-1}\geq x_{ij}\geq x_{i-1,j} &\mbox{ for all } 1\leq i \leq j\leq n.
\end{align*}

Recall also that two integral polytopes $\mathcal{P}$ in $\mathbb{R}^d$ and $\mathcal{Q}$ in $\mathbb{R}^m$ are \textbf{integrally equivalent}
if there is an affine transformation
$\varphi\colon\mathbb{R}^d \to \mathbb{R}^m$ whose restriction to
$\mathcal{P}$ is a bijection $\varphi\colon \mathcal{P} \to \mathcal{Q}$
that preserves the lattice, i.e., $\varphi$ is a
bijection between $\mathbb{Z}^d \cap \aff(\mathcal{P})$ and
$\mathbb{Z}^m \cap \aff(\mathcal{Q})$, where $\aff(\cdot)$ denotes affine span. The map $\varphi$ is called an \textbf{integral equivalence}. Note that integrally equivalent polytopes have the same Ehrhart polynomials, and therefore the same volume.

We now define the flow network $G_\lambda$, describing the graph and its associated netflow (see Remark \ref{rem:flownetwork}). For an illustration of $G_{\lambda}$, see  Figure \ref{Glambda}. 
\begin{definition}
	\label{def:Glambda}
	For a partition $\lambda\in\mathbb{Z}^n_{\geq 0}$ with $n\geq 2$, let $G_\lambda$ be defined as follows:
	
	\noindent If $n=1$, let $G_\lambda$ be a single vertex $v_{22}$ defined to have flow polytope consisting of one point, $0$. Otherwise, let $G_\lambda$ have vertices 
	\[V(G_\lambda) = \{v_{ij} \mid  2\leq i\leq j \leq n\}\cup\{v_{i,i-1} \mid  3\leq i\leq n+2\}\cup\{v_{i,n+1} \mid  3\leq i \leq n+1\} \]
	and edges 
	\begin{equation*}
		\begin{split}
		E(G_\lambda) &= \{(v_{ij},v_{i+1,j}) \mid 2\leq i\leq j\leq n \}\cup\{(v_{i,n+1},v_{i+1,n+1}) \mid 3\leq i \leq n+1 \}\\
		&\quad\cup\{(v_{ij},v_{i+1,j+1}) \mid  2\leq i\leq j\leq n \}\cup\{(v_{i,i-1},v_{i+1,i}) \mid  3\leq i \leq n+1 \}.
		\end{split}
	\end{equation*}
	The default netflow vector on $G_\lambda$ is as follows:
	\begin{itemize}
		\item To vertex $v_{2j}$ for $2\leq j \leq n$, assign netflow $\lambda_{j-1}-\lambda_{j}$.
		\item To vertex $v_{n+2,n+1}$, assign netflow $\lambda_{n}-\lambda_{1}$.
		\item To all other vertices, assign netflow $0$.
	\end{itemize}
	Given a flow on $G_\lambda$, denote the flow value on each edge $(v_{ij},v_{i+1,j})$ by $a_{ij}$, and denote the flow value on each edge $(v_{ij},v_{i+1,j+1})$ by $b_{ij}$.
\end{definition}

\begin{figure}[ht]
	\includegraphics[scale=.55]{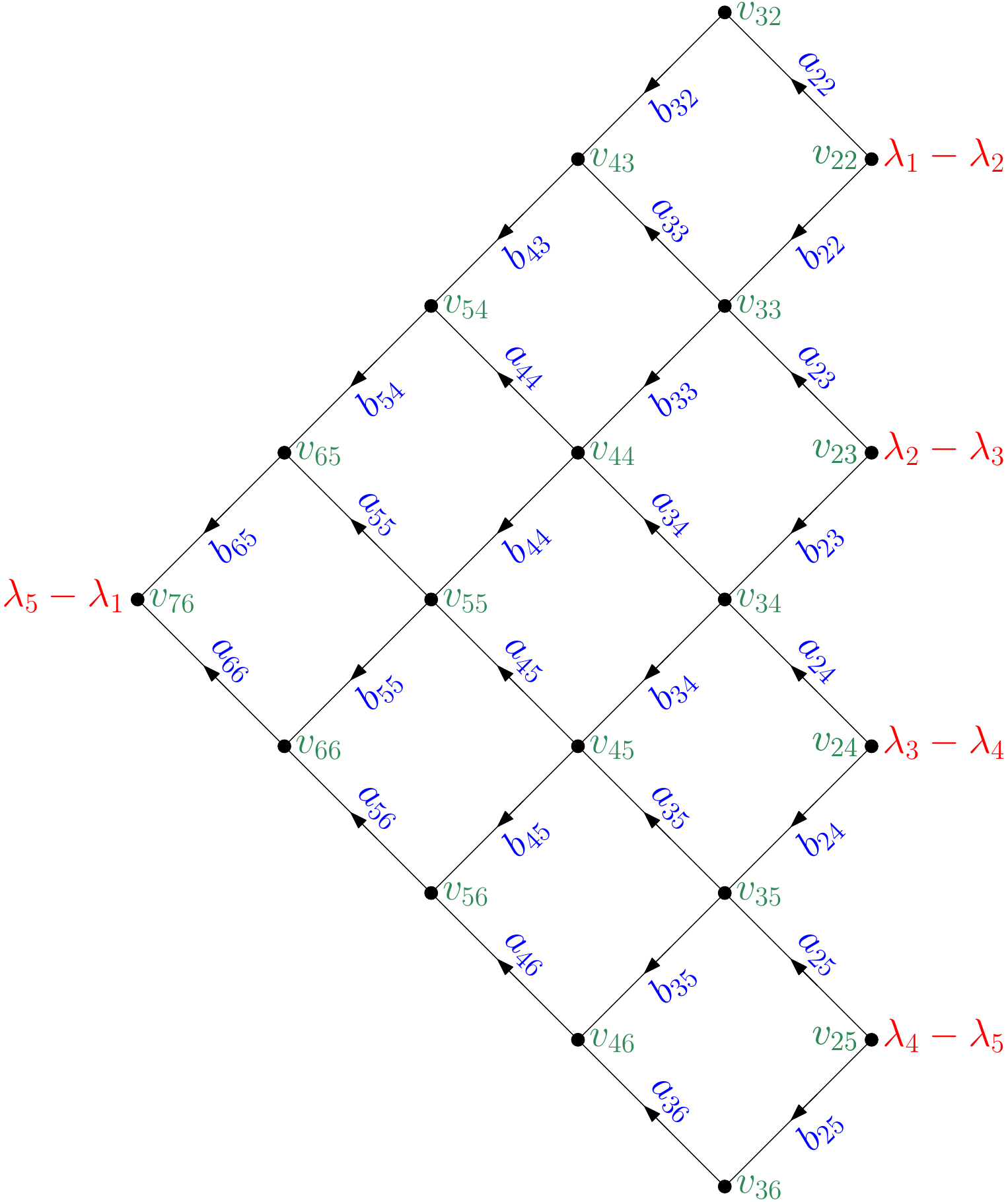}
	\caption{The flow network $G_\lambda$ with $\ell(\lambda)=5$.}
	\label{Glambda}
\end{figure} 
 
\bigskip
\begin{proof}[Proof of Theorem \ref{thm:gtflowpolytope}.]
	To map a point $(x_{ij})_{i,j}\in \mathrm{GT}(\lambda)$ to $\mathcal{F}_{G_{\lambda}}$, use the map
	\begin{align*}
	a_{i\,j}&=x_{i-1,j-1}-x_{ij},\\
	b_{i\,j}&=x_{ij}-x_{i-1,j}.
	\end{align*}
	Conversely, to map a flow $f\in\mathcal{F}_{G_\lambda}$ to $\mathrm{\mathrm{GT}(\lambda)}$, use either
	\[x_{ij}=\lambda_j+\sum_{k=2}^{i}b_{kj}\mbox{\hspace{2ex}or\hspace{2ex} } x_{ij}=\lambda_{j-i+1}-\sum_{k=0}^{i-2}a_{i-k,j-k}.\]
	It is easily checked these two maps are inverses of each other and are both integral, completing the proof.
\end{proof}

\begin{example}
	\label{exp:coordinates}
	For $n=5$, the integral equivalences between $\mathrm{GT}(\lambda)$ and $\mathcal{F}_{G_\lambda}$ are: \newline
	\begin{minipage}{.4\linewidth}
		\centering
		\includegraphics[scale=.40]{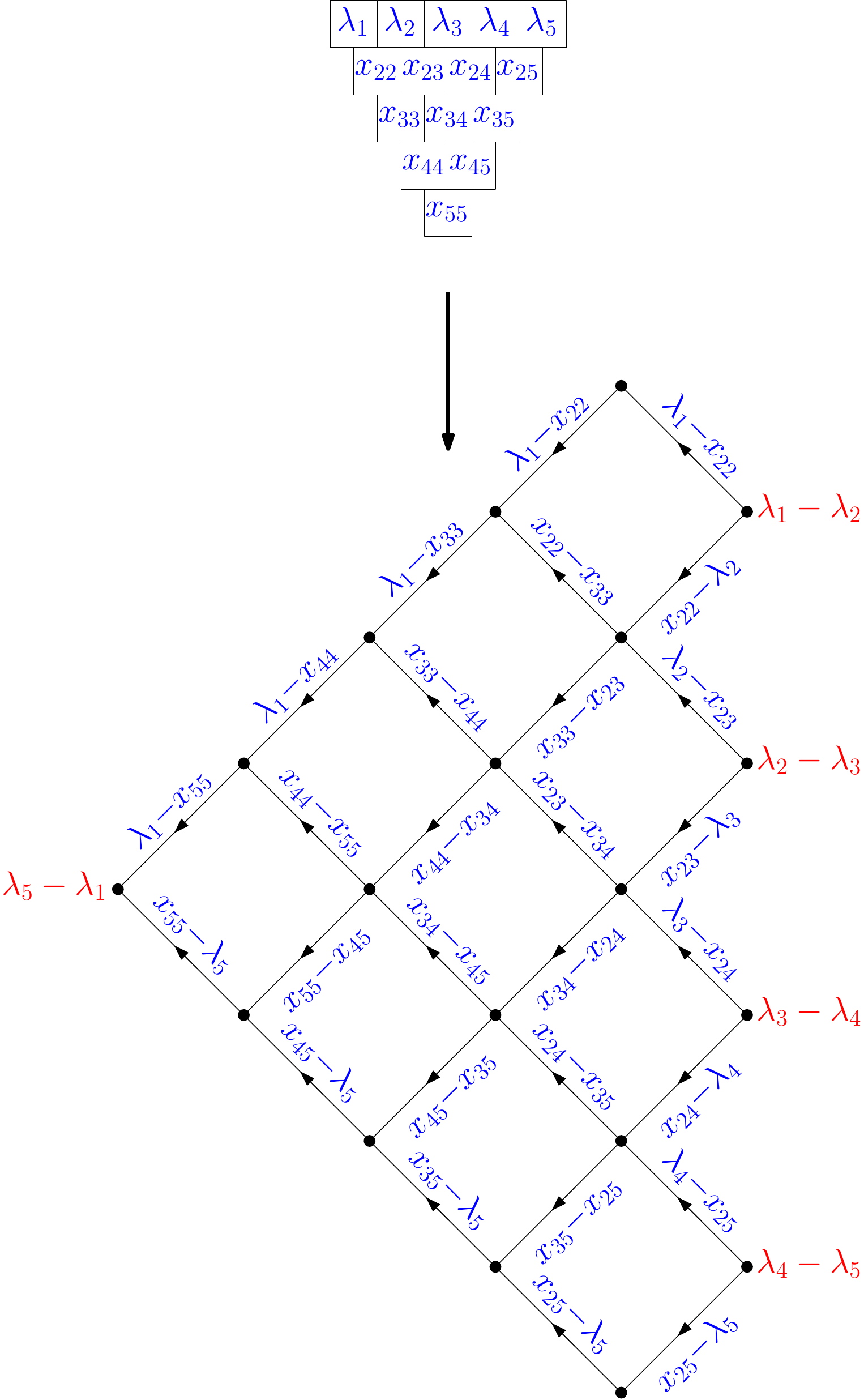}
	\end{minipage}%
	\begin{minipage}{0.5\linewidth}
		\centering
		\includegraphics[scale=.40]{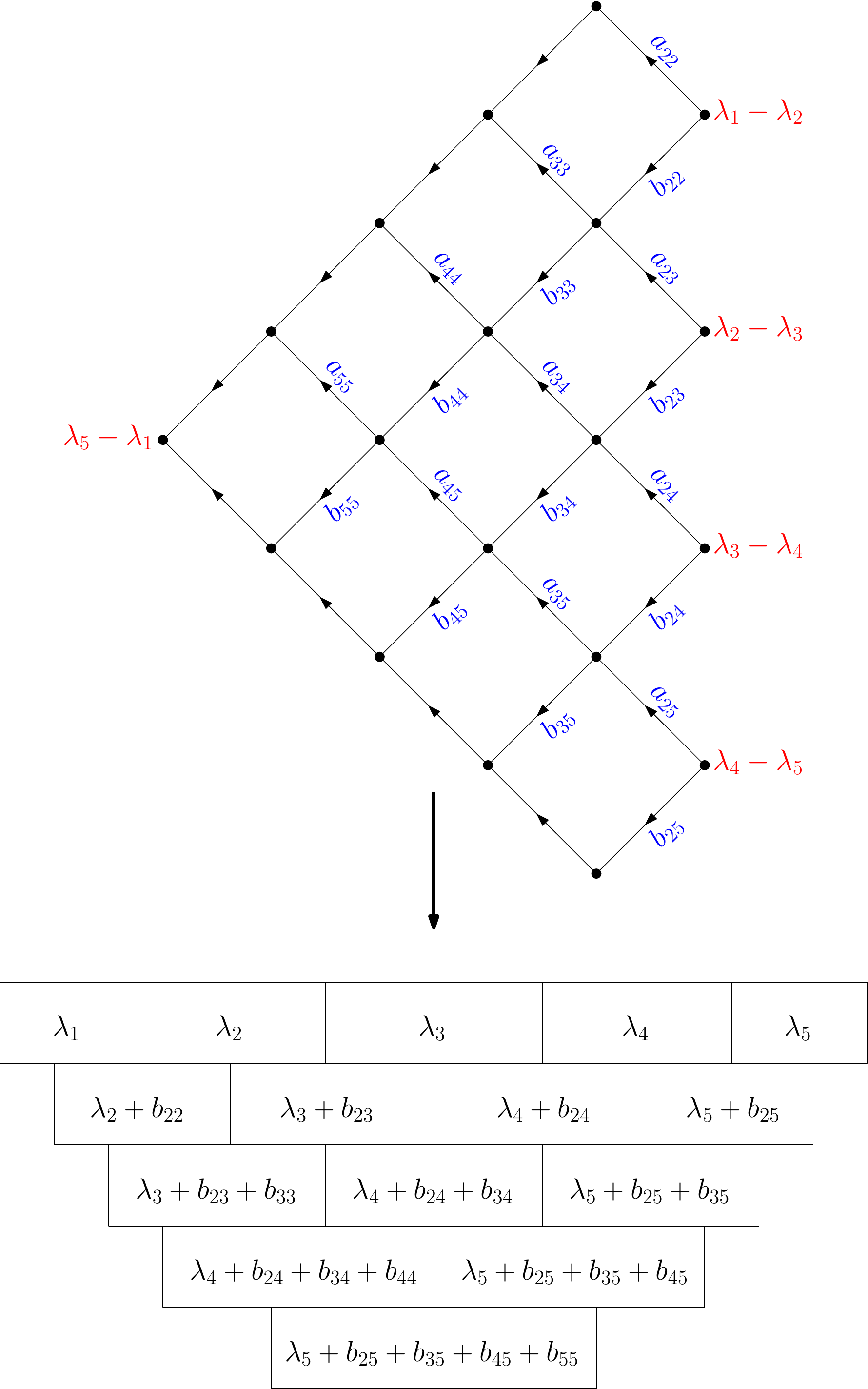}
	\end{minipage}
\end{example}

\subsection{Consequences of the Gelfand-Tsetlin polytope being a flow polytope} Here we provide a few corollaries to the Gelfand-Tsetlin polytope  $\mathrm{GT}(\lambda)$ being integrally equivalent to the flow polytope $\F_{G_{\lambda}}$. In \cite{paper2} we give further applications of this result, particularly about the volume and Ehrhart polynomial of Gelfand-Tsetlin polytopes. The corollaries presented below are all well-known; we include them here to demonstrate proofs via flow polytopes. We begin with two well-known results about flow polytopes, and then we give their applications to Gelfand-Tsetlin polytopes.

\begin{lemma}[\cite{BV2}]
	\label{prop:flowminkowskidecomposition}
	For a graph $G$ on $[n+1]$ and nonnegative integers $a_1,\ldots,a_n$,
	\[\mathcal{F}_G(a_1,\ldots,a_n,-\sum_{i=1}^{n}a_i) = a_1\mathcal{F}_G(e_1-e_{n+1})+a_2\mathcal{F}_G(e_2-e_{n+1})+\cdots+a_n\mathcal{F}_G(e_n-e_{n+1}).\]
\end{lemma}
\begin{proof}
	One inclusion is proven by adding flows edgewise. The other is shown by induction on the number of nonzero $a_i$.
\end{proof}

\begin{corollary}
	If $G$ is a graph on $[n+1]$ and $a_1,\ldots,a_n,b_1,\ldots,b_n$ are nonnegative integers, then 
	\[\mathcal{F}_G(a_1,\ldots,a_n,-\sum_{i=1}^{n}a_i)+\mathcal{F}_G(b_1,\ldots,b_n,-\sum_{i=1}^{n}b_i) = \mathcal{F}_G(a_1+b_1,\ldots,a_n+b_n,-\sum_{i=1}^{n}a_i+b_i). \]
\end{corollary}
\begin{proof}
	Induct on the number of nonzero $b_i$ and use Lemma \ref{prop:flowminkowskidecomposition}.
\end{proof}

As a consequence of the previous two results and the integral equivalence of $\mathrm{GT}(\lambda)$ and $\mathcal{F}_{G_\lambda}$, we obtain the following two well-known facts about Gelfand-Tsetlin polytopes.
\begin{namedlemma}[\ref{lem:gtsum}]  
	If $\lambda$ is a partition with $n$ parts, then the Gelfand-Tsetlin polytope $\mathrm{GT}(\lambda)$ decomposes as the Minkowski sum
	\[\mathrm{GT}(\lambda) = \sum_{k=1}^{n}(\lambda_k-\lambda_{k+1})\mathrm{GT}(1^k0^{n-k}),\]
	where $\lambda_{n+1}$ is taken to be zero.
\end{namedlemma}

\begin{lemma}
	If $\lambda$ and $\mu$ are partitions with $n$ parts, then 
	\[\mathrm{GT}(\lambda)+\mathrm{GT}(\mu) = \mathrm{GT}(\lambda+\mu).\]
\end{lemma}

\vspace{2ex}
\noindent Recall that the \textbf{Schur polynomial} $s_\lambda$ can be expressed as
\[s_{\lambda}(x_1,\ldots,x_n)=\sum_{P\in \mathrm{GT}(\lambda)\cap \mathbb{Z}^{\binom{n+1}{2}}} x_1^{wt(P)_1}x_2^{wt(P)_2}\cdots x_n^{wt(P)_n} \]
where $wt\colon\mathbb{R}^{\binom{n+1}{2}}\to\mathbb{R}^n$ is the \textbf{weight map}, defined by 
\[wt(P)_i = \sum_{j=i}^{n}x_{ij}-\sum_{j=i+1}^{n} x_{i+1,j}\]
for $P\in \mathrm{GT}(\lambda)$.
We now introduce the flow polytopal analogue of $wt$ and study it. Recall the variables $\{a_{ij}\}_{i,j}\cup\{b_{ij}\}_{i,j}$ of Definition \ref{def:Glambda}: in $\mathcal{F}_{G_\lambda}$, $a_{ij}$ represents the flow on the edge $(v_{ij},v_{i+1,j})$ and $b_{ij}$ represents the flow on the edge  $(v_{ij},v_{i+1,j+1})$.

\begin{definition}
	Let $\lambda$ be a partition with $n$ parts. Define the \textbf{graphical weight map} $gwt\colon\mathbb{R}^{E(G_\lambda)}\to \mathbb{R}^n$ by setting
	\[gwt(x_{(v_{ij},v_{i+1,j})}) =e_{i-1}\mbox{\hspace{2ex} and \hspace{1ex}} gwt(x_{(v_{ij},v_{i+1,j+1})}) =0, \]
	so in particular
	\[gwt(a_{ij}) =e_{i-1}\mbox{\hspace{2ex} and \hspace{1ex}} gwt(b_{ij}) =0. \]
\end{definition}

\begin{proposition}
	\label{prop:graphicalweightshift}
	For a partition $\lambda$ with $n$ parts, let $f\in\mathcal{F}_{G_{\lambda}}$ correspond to $P_f\in \mathrm{GT}(\lambda)$. Then, the maps $gwt$ and $wt$ are related by the translation
	\[wt(P_f)=gwt(f) +\lambda_n \bm{1}_n, \]
	where $\bm{1}_n$ denotes the vector of all ones in $\mathbb{R}^n$.
\end{proposition}
\begin{proof}
	We have
	\begin{align*}
	gwt(f)_i&=a_{i+1,i+1}+\cdots+a_{i+1,n}+a_{i+1,n+1}\\
	&=a_{i+1,i+1}+\cdots+a_{i+1,n}+b_{2n}+\cdots+b_{in}.
	\end{align*}
	Using the integral equivalence $x_{ij}=\lambda_{j-i+1}-\sum_{k=0}^{i-2}a_{i-k,j-k}$ between $\mathrm{GT}(\lambda)$ and $\mathcal{F}_{G_\lambda}$,
	\begin{align*}
	wt(P_f)_i&=\sum_{j=i}^{n}{x_{ij}}-\sum_{j=i+1}^{n}x_{i+1,j}\\
	&=x_{in}+\sum_{j=i}^{n-1}\left(x_{ij}-x_{i+1,j+1}\right)\\
	&=x_{in}+\sum_{j=i+1}^{n}a_{i+1,j}.
	\end{align*}
	Now, using the integral equivalence $x_{ij}=\lambda_j+\sum_{k=2}^{i}b_{kj}$, we have
	\begin{align*}
	(gwt(f)-wt(P_f))_i&=x_{in}-\sum_{k=2}^{i}b_{kn}\\
	&=\left(\lambda_n+\sum_{k=2}^{i}b_{kn}\right)-\sum_{k=2}^{i}b_{kn}\\
	&=\lambda_n.\qedhere
	\end{align*}
\end{proof}

Using the map $gwt$, we now describe the polytopes $\mathrm{GT}(1^k0^{n-k})$ and rederive a result of Postnikov from \cite{beyond}.
\begin{proposition}
	\label{prop:gwt}
	If $\lambda$ is of the form $1^k0^{n-k}$ with $1\leq k\leq n$, then $gwt(\mathcal{F}_{G_\lambda})$ equals the hypersimplex $\Delta_{k,n}=\mathrm{Conv}(\{x\in[0,1]^n \mid x_1+x_2+\cdots+x_n=k\})$. 
\end{proposition}
\begin{proof}
	If $\lambda$ is of the form $1^k0^{n-k}$, then $G_\lambda$ will have a single source with netflow $1$ and a single sink with netflow $-1$. Ignoring all edges and vertices not lying on path from the source to sink (which will carry zero flow), we are left with a rectangular grid as shown in Figure \ref{fig:rectgwt}. A path from source to sink in the grid requires $k$ NW steps and $n-k$ SW steps. Recall (cf. \cite{qcp}, Lemma 3.1) that the vertices of a flow polytope with a single source and sink are exactly the flows that are nonzero only on a path from source to sink.
	
	Thus, the vertices of $\mathcal{F}_{G_\lambda}$ are exactly the flows with support a path from source to sink in the grid. These paths are in bijection with length $n$ words on $\{N,S\}$ having $k$ $N$'s (corresponding to NW steps in the path) and $n-k$ $S$'s (corresponding to SW steps in the path). By definition, the map $gwt$ takes a vertex of $\mathcal{F}_{G_\lambda}$ to the vector with ones in the positions of the $N$'s in the corresponding string, and zero elsewhere. Thus, 
	\[gwt(V(\mathcal{F}_{G_\lambda})) = \{x\in \{0,1\}^n \mid  x_1+\cdots+x_n=k \} = V(\Delta_{k,n}),\]
	so $gwt(\mathcal{F}_{G_\lambda})=\Delta_{k,n}$.
\end{proof}
\begin{figure}
		\begin{center}
			\includegraphics[scale=.45]{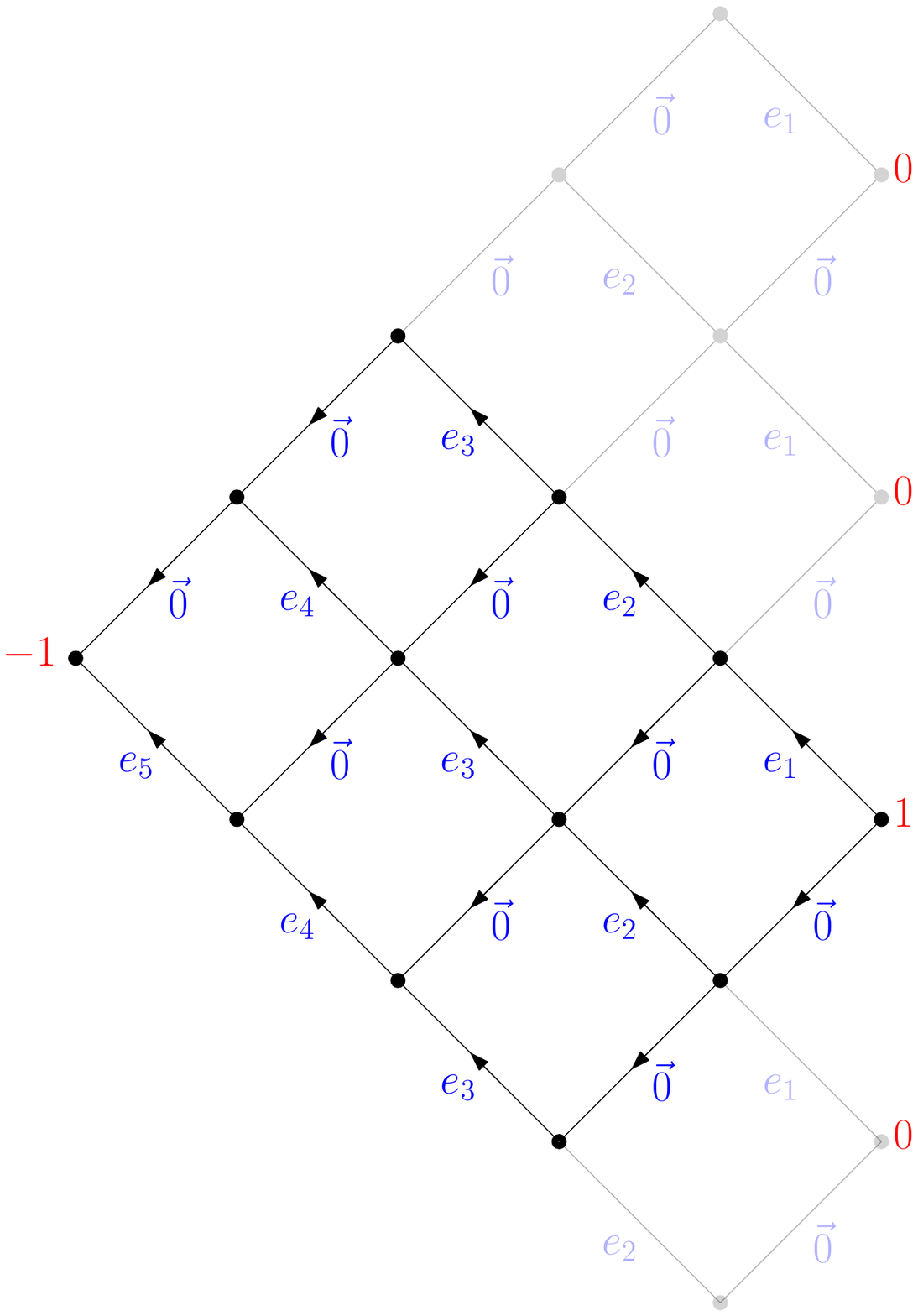}
		\end{center}
		\caption{$\mathrm{GT}(1,1,1,0,0)$ and the associated map $gwt$.}
		\label{fig:rectgwt}
	\end{figure}
 
\begin{corollary}[\cite{beyond}]
	The permutahedron $\mathcal{P}_\lambda=\mathrm{Conv}(S_n\cdot \lambda)$ of $\lambda$ equals the Minkowski sum of hypersimplices
	\[\mathcal{P}_\lambda = (\lambda_1-\lambda_2)\Delta_{1,n}+(\lambda_2-\lambda_3)\Delta_{2,n}+\cdots+(\lambda_{n-1}-\lambda_n)\Delta_{n-1,n}+\lambda_{n}\Delta_{n,n}. \] 
\end{corollary}
\begin{proof}
	Since $wt(\mathrm{GT}(\lambda)) = \mathcal{P}_\lambda$, applying $gwt$ to both sides of 
	\[\mathcal{F}_{G_\lambda}=\sum_{k=1}^{n-1} (\lambda_k-\lambda_{k+1})\mathcal{F}_{G_{(1^k0^{n-k})}} \]
	and using Propositions \ref{prop:gwt} and \ref{prop:graphicalweightshift} yields 
	\[\mathcal{P}_\lambda - \lambda_n\bm{1}_n = (\lambda_1-\lambda_2)\Delta_{1,n}+(\lambda_2-\lambda_3)\Delta_{2,n}+\cdots+(\lambda_{n-1}-\lambda_n)\Delta_{n-1,n}.\qedhere \]
\end{proof}

\bigskip

\subsection{The Minkowski sum of Gelfand-Tsetlin polytopes}
\label{sec:sum}
In this section we observe that the Minkowski sum of Gelfand-Tsetlin polytopes $\p_D$ appearing in Theorem \ref{thm:gtsum} can be viewed naturally as a subset of a larger Gelfand-Tsetlin polytope. 

Recall the embedding of the Gelfand-Tsetlin polytopes in the sum $\p_D=\mathrm{GT}(\lambda^{(1)})+\mathrm{GT}(\lambda^{(2)})+\cdots+\mathrm{GT}(\lambda^{(n)})$ from Section \ref{sec:intro}. In light of Theorem \ref{thm:gtflowpolytope}, $\p_D$ should be integrally equivalent to a sum of flow polytopes  \[\mathcal{F}_{G_{\lambda^{(1)}}}+\cdots+\mathcal{F}_{G_{\lambda^{(n)}}}.\] Just like for the Gelfand-Tsetlin polytope sum, we must specify how the graphs $G_{\lambda^{(i)}}$, $i \in [n]$, are embedded. 
Let us  embed $G_{\lambda^{(k)}}$, $k \in [n]$,  into $G_{\lambda^{(n)}}$ by identifying $v_{ij}$ (see Definition \ref{def:Glambda}) in $G_{\lambda^{(k)}}$ with $v_{i,j+n-k}$ in $G_{\lambda^{(k)}}$. Note that the trivial case $G_{\lambda^{(1)}}$ is just a single vertex with netflow $0$ and flow polytope defined to be the single point $0$.

Lemmas \ref{lem:ms} and \ref{lem:g} follow readily by the definitions and the integral equivalence given in Theorem \ref{thm:gtflowpolytope}:

\begin{lemma} \label{lem:ms}
	The Minkowski sum 
	\[\mathrm{GT}(\lambda^{(1)})+\cdots+\mathrm{GT}(\lambda^{(n)}) \]
	is integrally equivalent to 
	\[\mathcal{F}_{G_{\lambda^{(1)}}}+\cdots+\mathcal{F}_{G_{\lambda^{(n)}}} \]  with the embedding specified above.
\end{lemma}

\begin{definition}
	Given partitions $\lambda^{(k)}$ of size $k$ for $k\in[n]$, let $G({\lambda^{(1)}},\ldots,{\lambda^{(n)}})$ denote the flow network obtained by overlaying the flow networks $G_{\lambda^{(1)}},\ldots,G_{\lambda^{(n)}}$ according to the embedding specified above and adding the corresponding netflows. Let $\widehat{G}({\lambda^{(1)}},\ldots,{\lambda^{(n)}})$ denote the flow network obtained from  $G_{\lambda^{(1)}},\ldots,G_{\lambda^{(n)}}$ by moving all negative netflows to $v_{n+2,n+1}$ and replacing them by zero netflows. The case $n=4$ is demonstrated in Figure \ref{fig:sumexample}.	
\end{definition}

\begin{figure}[ht]
	\begin{center}
		\includegraphics[scale=.8]{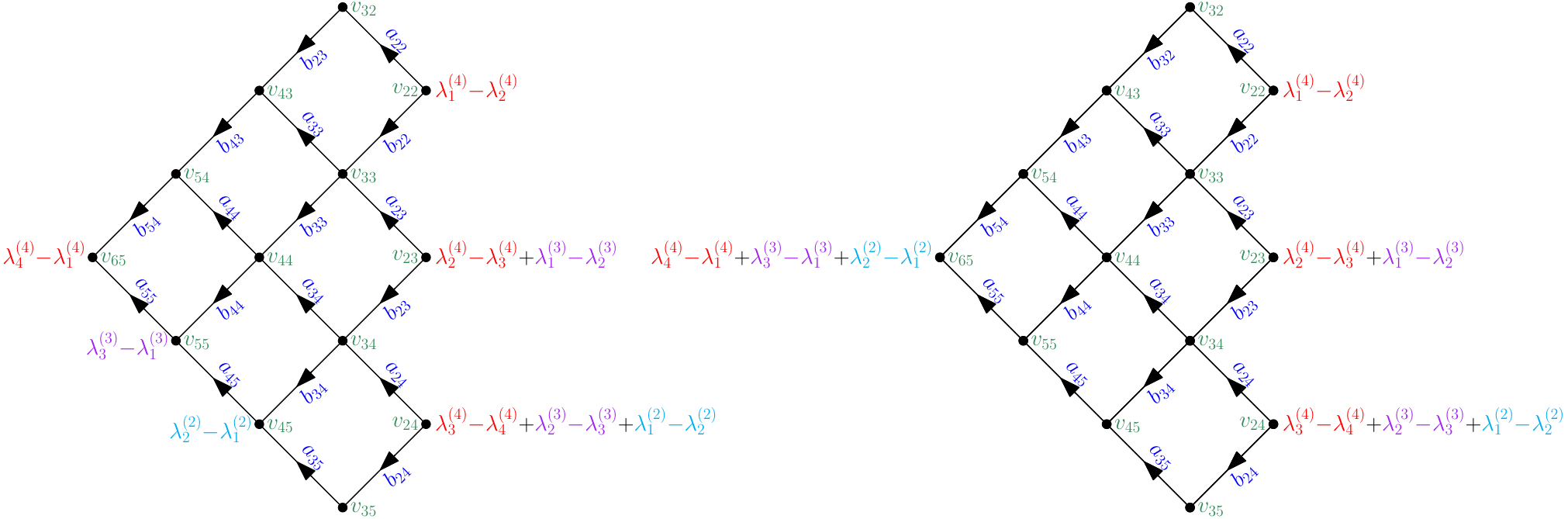}
	\end{center}
	\caption{The flow networks $G({\lambda^{(1)}},{\lambda^{(2)}},{\lambda^{(3)}},{\lambda^{(4)}})$ (left) and $\widehat{G}({\lambda^{(1)}},{\lambda^{(2)}},{\lambda^{(3)}},{\lambda^{(4)}})$ (right)}
	\label{fig:sumexample}
\end{figure}

\begin{lemma}\label{lem:g} The following polytope inclusions hold:
	\[\mathcal{F}_{G_{\lambda^{(n)}}}+\cdots+\mathcal{F}_{G_{\lambda^{(1)}}} \subset \mathcal{F}_{G(\lambda^{(1)},\ldots,\lambda^{(n)})} \subset \mathcal{F}_{\widehat{G}(\lambda^{(1)},\ldots,\lambda^{(n)})},\] 
	the latter being true up to an integral translation of $\mathcal{F}_{G(\lambda^{(1)},\ldots,\lambda^{(n)})}$.
	
	\medskip
	
	\noindent	In general, none of the above inclusions is an equality. The polytope $\mathcal{F}_{\widehat{G}(\lambda^{(1)},\ldots,\lambda^{(n)})}$ is integrally equivalent to the Gelfand-Tsetlin polytope ${\rm GT }(\mu)$ where $\mu_n$ is arbitrary, and for $k<n$,
	\[\mu_k=\mu_{k+1}+\sum_{j=0}^{k-1} \lambda_{k-j}^{(n-j)}-\lambda_{k-j+1}^{(n-j)}. \] 
\end{lemma}

Thus, we conclude that for a column-convex diagram $D$  the polytope  $\p_D$ can be thought of as obtained from ${\rm GT }(\mu)$ specified in Lemma \ref{lem:g} via further hyperplane cuts. Recall also Proposition   \ref{prop:gtsum}, which gives another view on $\p_D$.
 
  \section*{Acknowledgments}
  We are grateful to  Allen Knutson for inspiring conversations about Schubert polynomials. 
 
\bibliographystyle{plain}
\bibliography{schubertpolynomialsasprojections}
\end{document}